\documentclass[11pt,a4paper]{article}

\usepackage[svgnames]{xcolor}

\usepackage{amsmath, amssymb, amsthm, amsfonts}

\usepackage[all]{xy}

\usepackage[margin=2.5cm]{geometry}
\usepackage{enumitem}

\usepackage{enumitem}
\setenumerate[1]{itemsep=0pt,partopsep=0pt,parsep=\parskip,topsep=3pt}
\setitemize[1]{itemsep=0pt,partopsep=0pt,parsep=\parskip,topsep=3pt}
\setdescription{itemsep=0pt,partopsep=0pt,parsep=\parskip,topsep=3pt}
\setlist[itemize]{leftmargin=35pt}
\setlist[enumerate]{leftmargin=35pt}

\theoremstyle{plain}
\newtheorem{theorem}{Theorem}[section]
\newtheorem{proposition}[theorem]{Proposition}
\newtheorem{lemma}[theorem]{Lemma}

\theoremstyle{definition}
\newtheorem{definition}[theorem]{Definition}
\newtheorem{example}[theorem]{Example}  
\newtheorem{remark}[theorem]{Remark}

\DeclareMathOperator{\module}{mod}
\renewcommand{\mod}{\module} 
\DeclareMathOperator{\Hom}{Hom}

\DeclareMathOperator{\End}{End}

\usepackage{hyperref}
\hypersetup{
    colorlinks=true,   
    citecolor=red,     
    linkcolor=blue,    
    urlcolor=black     
}

\begin{document}
\title{\bf Projectively Wakamatsu Tilting Modules over One-Point Extensions}
\author{
Dajun Liu$^{1}$, Jiaxuan Feng$^{2}$, Hanpeng Gao$^{3,*}$\\
{\footnotesize  $^{1,2}$School of Mathematics-Physics and Finance, Anhui Polytechnic University,  Wuhu 241000 Anhui, P. R. China;}\\
{\footnotesize  $^{3}$School of Mathematical Sciences, Anhui University, Hefei 230601, P. R. China;}\\
{\footnotesize 1. E-mail: liudajun@ahpu.edu.cn; ORCID: \href{https://orcid.org/0009-0001-6073-7587}{0009-0001-6073-7587}; }\\
{\footnotesize 2. E-mail: 1254708283@qq.com;}\\
{\footnotesize 3. E-mail: hpgao@ahu.edu.cn.}}
\date{}
\maketitle
\footnotetext{
Supported by the National Natural Science Foundation of China (No.12101003,12301041), the Natural Science Foundation of Anhui province (No.2108085QA07).\\
* Corresponding author.\\  
}

\begin{abstract}
Let $\Gamma = \Lambda[M]$ be the one-point extension of an algebra $\Lambda$ by a $\Lambda$-module $M$. We establish a  method to lift projectively Wakamatsu tilting (PWT) modules from $\mathrm{mod}\,\Lambda$ to $\mathrm{mod}\,\Gamma$ by adding the new projective module, and prove that this lifting process perfectly preserves  mutation relations under certain homological conditions. Furthermore, for source point extensions of representation-finite algebras,   we obtain a complete classification of PWT $\Gamma$-modules in terms of those over $\Lambda$. In particular, we establish a bijection
\[
\mathrm{PWT}(\Gamma) \longleftrightarrow \mathrm{PWT}(\Lambda) \coprod \mathrm{RPWT}(\Lambda, S_i).
\]
which yields the counting formula about $|\mathrm{PWT}(\Gamma)|$.

\end{abstract}

\textbf{Keywords}: projectively Wakamatsu tilting modules, one-point extensions, representation-finite algebras,Wakamatsu tilting modules.

\textbf{MR(2020) Subject Classification}: 16G10
\maketitle

\section{Introduction}

Tilting theory plays a central role in the representation theory of Artin algebras. 
The notion of a tilting module has been extensively generalized and developed in various directions. One of the most prominent generalizations is the $\tau$-tilting theory introduced by Adachi, Iyama, and Reiten \cite{AIR14}, which completes the classical tilting mutation and connects closely to cluster algebras. Another important generalization is given by Wakamatsu tilting modules, introduced by Wakamatsu~\cite{Wa90}, 
which provide a powerful framework in the study of derived equivalences and Gorenstein homological algebra.

Recently, Enomoto \cite{Eno23} introduced a subclass of Wakamatsu tilting modules called \emph{projectively Wakamatsu tilting} (PWT) modules. A self-orthogonal module $T$ is projectively Wakamatsu tilting if it is exactly an $\mathrm{Ext}$-progenerator for its right perpendicular category $T^{\perp_{>0}}$. Enomoto proved that for representation-finite algebras, PWT modules coincide with Wakamatsu tilting modules as well as with maximal self-orthogonal modules. 
This result provides a new perspective on the structure of self-orthogonal modules and motivates further investigation of PWT modules.

One-point extensions of algebras form a classical and powerful inductive tool for constructing new algebras from given ones. The behavior of tilting modules under one-point extensions was extensively studied by Assem, Happel, and Trepode \cite{AHT07}. More recently, Suarez \cite{Sua18} generalized this approach to support $\tau$-tilting modules using restriction and extension functors. Motivated by these developments, this paper aims to study the construction of PWT modules over one-point extensions.

Let $\Lambda$ be an algebra and $M \in \mathrm{mod}\,\Lambda$. Let $\Gamma = \Lambda[M]$ be the one-point extension algebra of $\Lambda$ by $M$. The module categories $\mathrm{mod}\,\Lambda$ and $\mathrm{mod}\,\Gamma$ are connected by a pair of adjoint functors: the exact restriction functor $\mathcal{R}$ and the fully faithful extension functor $\mathcal{E}$. Let $P_a$ and $S_a$ denote the new indecomposable projective and simple $\Gamma$-modules, respectively, associated with the extension vertex $a$.

Our first main result provides a method to lift a PWT module from $\Lambda$ to $\Gamma$.

\vspace{0.2cm}
\noindent \textbf{Theorem A.} (See Theorem \ref{thm:main_construction_proj}) { Let $\Gamma = \Lambda[M]$ be the one-point extension algebra of $\Lambda$ by $M$ and $U$ be a PWT module in $\mathrm{mod}\,\Lambda$.  If $M \in U^{\perp_\Lambda}$, then $\widetilde{U} = \mathcal{E}U \oplus P_a$ is a PWT module in $\mathrm{mod}\,\Gamma$.}
\vspace{0.2cm}

We further show that mutation relations between PWT modules are preserved under suitable homological conditions.

\vspace{0.2cm}
\noindent \textbf{Theorem B.} (See Theorem \ref{thm:mutation_preservation}) {Let $\Gamma = \Lambda[M]$ be the one-point extension of $\Lambda$ by $M \in \operatorname{mod} \Lambda$. 
	Suppose $U$  is a PWT module over $\Lambda$ and $U^{\prime}$ is a left mutation of $U$. If $M\in U^{\perp_{\Lambda}}\cap (U^{\prime})^{\perp_{\Lambda}}$, then
	$\widetilde{U}^{\prime}=\mathcal{E}U^{\prime}\oplus P_{a}$ is a left mutation of $\widetilde{U}=\mathcal{E}U\oplus P_{a}$ in $\mathrm{mod}\,\Gamma$.}

\vspace{0.2cm}
Finally, for source point extensions of representation-finite algebras, we obtain a complete classification of PWT modules.

\vspace{0.2cm}
\noindent \textbf{Theorem C.} (See Theorem \ref{thm:bijection}) {Let $\Gamma = \Lambda[S_i]$ be a source point extension of a representation-finite Artin algebra $\Lambda$. Then there exists a  bijection:}
\[
\mathrm{PWT}(\Gamma) \longleftrightarrow \mathrm{PWT}(\Lambda) \coprod \mathrm{RPWT}(\Lambda, S_i).
\]
{In particular, we have}
\[
|\mathrm{PWT}(\Gamma)| = |\mathrm{PWT}(\Lambda)| + |\mathrm{RPWT}(\Lambda, S_i)|.
\]

This result reveals a precise combinatorial and homological relationship between PWT modules over $\Lambda$ and those over its one-point extensions.

Throughout this paper, let $k$ be an algebraically closed field and let $\Lambda$ be a finite-dimensional $k$-algebra. We denote by $\operatorname{mod} \Lambda$ the category of finitely generated right $\Lambda$-modules.
For a module $M \in \operatorname{mod} \Lambda$, we denote by $\operatorname{add} M$ the full subcategory of $\operatorname{mod} \Lambda$ consisting of all direct summands of finite direct sums of $M$. We write $\operatorname{Ext}^{>0}_\Lambda(X, Y) = 0$ to indicate that $\operatorname{Ext}^i_\Lambda(X, Y) = 0$ for all integers $i > 0$.

The paper is organized as follows. 
In Section~\ref{sec2}, we recall basic notions on PWT modules and one-point extensions. 
In Section~\ref{sec3}, we establish the main results, including lifting constructions, mutation preservation, and classification results. 
In Section~\ref{sec4}, we provide examples to illustrate our theory.

\section{Preliminaries}\label{sec2}

In this section, we recall the basic definitions and fundamental properties of PWT modules and one-point extensions of algebras. 

\subsection{Projectively Wakamatsu Tilting Modules}

We begin by recalling the relevant concepts from Wakamatsu tilting theory, recently refined by Enomoto \cite{Eno23}.

A $\Lambda$-module $T$ is called \emph{self-orthogonal} if $\operatorname{Ext}^{>0}_\Lambda(T, T) = 0$. For a self-orthogonal module $T$,
\[
T^{\perp_\Lambda} = \{ X \in \operatorname{mod} \Lambda \mid \operatorname{Ext}^{>0}_\Lambda(T, X) = 0 \}.
\]
By standard homological algebra, $T^{\perp}$ is a coresolving subcategory of $\operatorname{mod} \Lambda$, meaning it is closed under extensions, direct summands, and cokernels of monomorphisms. Dually, we can define $^{\perp_\Lambda}T$.

Let $\mathcal{C}$ be a full subcategory of $\operatorname{mod} \Lambda$ closed under extensions and direct summands. An object $P \in \mathcal{C}$ is said to be \emph{Ext-projective} in $\mathcal{C}$ if $\operatorname{Ext}^1_\Lambda(P, \mathcal{C}) = 0$. We say that $P$ is an \emph{Ext-progenerator} of $\mathcal{C}$ if $P$ is Ext-projective in $\mathcal{C}$ and for every object $X \in \mathcal{C}$, there exists a short exact sequence in $\operatorname{mod} \Lambda$:
\[
0 \longrightarrow K \longrightarrow P_0 \longrightarrow X \longrightarrow 0
\]
with $P_0 \in \operatorname{add} P$ and $K \in \mathcal{C}$.

With these notions, we recall the central definition of this paper:

\begin{definition}\label{def:tilting_wakamatsu}
Let $T$ be a $\Lambda$-module.
\begin{enumerate}[label=(\roman*)]
    \item $T$ is called a \emph{tilting module} if it satisfies three conditions: 
    
    (1) its projective dimension is finite ($\operatorname{pd}_\Lambda T < \infty$);
    
     (2) $T$ is self-orthogonal; 
     
     (3) there exists a finite exact sequence
    \[
    0 \longrightarrow \Lambda \longrightarrow T_0 \longrightarrow T_1 \longrightarrow \dots \longrightarrow T_d \longrightarrow 0
    \]
    with $T_i \in \operatorname{add} T$ for all $0 \le i \le d$.

    \item Let $T$ be a self-orthogonal $\Lambda$-module. $T$ is called \emph{projectively Wakamatsu tilting} (PWT) if $T$ is an Ext-progenerator of $T^{\perp}$.

    \item $T$ is called a \emph{Wakamatsu tilting module} if $T$ is self-orthogonal and there exists an infinite exact sequence
    \[
    0 \longrightarrow \Lambda \xrightarrow{f_0} T_0 \xrightarrow{f_1} T_1 \xrightarrow{f_2} \dots
    \]
    with $T_i \in \operatorname{add} T$, such that $\operatorname{Ext}^{>0}_\Lambda(\operatorname{Im} f_i, T) = 0$ for all $i \ge 0$.
\end{enumerate}
\end{definition}

\begin{remark}
The definition of PWT modules is deeply motivated by a foundational theorem of Auslander and Reiten \cite[Theorem 5.5]{AR91}. They established that if $T$ is a classical tilting module (and thus has finite projective dimension), then its perpendicular category $T^{\perp_\Lambda}$ is covariantly finite, and $T$ is necessarily an Ext-progenerator of $T^{\perp_\Lambda}$. For a Wakamatsu tilting module $T$, let $\mathcal{Y}_T$ be the subcategory consisting of modules $Y$ that admit an infinite exact sequence $\dots \to T_1 \to T_0 \to Y \to 0$ with $T_i \in \operatorname{add} T$ such that all kernels belong to $T^{\perp_\Lambda}$. For a PWT module $T$, the equality $\mathcal{Y}_T = T^{\perp_\Lambda}$ holds. Since injective modules lie in $T^{\perp_\Lambda}$, the injective cogenerator $D\Lambda$ belongs to $\mathcal{Y}_T$, which precisely satisfies the definition of a Wakamatsu tilting module (see \cite[Proposition 2.2]{BS98} or \cite[Proposition 2.8]{Eno23}). So,
\[
\{ \text{Tilting} \} \subseteq \{ \text{Projectively Wakamatsu tilting} \} \subseteq \{ \text{Wakamatsu tilting} \}.
\]
\end{remark}

An Artin algebra $\Lambda$ is said to be \emph{representation-finite} if there exist only finitely many isomorphism classes of finitely generated indecomposable $\Lambda$-modules.

\begin{lemma}[{\cite[Theorem 3.22]{Eno23}}]\label{lem:pwt_max_orthogonal}
Let $\Lambda$ be a representation-finite Artin algebra. A $\Lambda$-module $T$ is PWT if and only if $T$ is a maximal self-orthogonal module (i.e., $T$ is self-orthogonal and $|T| = |\Lambda|$) if and only if $T$ is a Wakamatsu tilting module.
\end{lemma}

\begin{remark}\label{rem:proj_inj_summands}
If $T$ is a PWT module, any indecomposable projective-injective module $I$ (i.e. $\operatorname{Ext}^{>0}_\Lambda(I, X) = 0$ and $\operatorname{Ext}^{>0}_\Lambda(X, I) = 0$ for any module $X \in \operatorname{mod} \Lambda$) must appear as a direct summand of $T$.
\end{remark}

\subsection{One-Point Extensions}

Let $M \in \operatorname{mod} \Lambda$. The \emph{one-point extension} of $\Lambda$ by $M$ is defined as the matrix algebra:
\[
\Gamma = \Lambda[M] = \begin{pmatrix} \Lambda & M \\ 0 & k \end{pmatrix},
\]
with the ordinary matrix addition and the multiplication induced by the $\Lambda$-module structure of $M$.

The category $\operatorname{mod} \Gamma$ can be identified with the category of triples $(X, V, \phi)$, where $X \in \operatorname{mod} \Lambda$, $V$ is a finite-dimensional $k$-vector space, and $\phi \colon V \otimes_k M \to X$ is a $\Lambda$-homomorphism. We denote by $a$ the new extension vertex. Up to isomorphism, there is a unique indecomposable projective $\Gamma$-module $P_a$ that is not a $\Lambda$-module, and it corresponds to the triple $(M, k, \operatorname{id}_M)$. Let $S_a = (0, k, 0)$ be the corresponding simple $\Gamma$-module. There exists a canonical short exact sequence in $\operatorname{mod} \Gamma$:
\begin{equation} \label{eq:canonical_seq}
0 \longrightarrow M \longrightarrow P_a \longrightarrow S_a \longrightarrow 0,
\end{equation}
where $M$ is naturally viewed as a $\Gamma$-module concentrated at the subalgebra $\Lambda$. Since $a$ is a source vertex in the ordinary quiver of $\Gamma$, the simple module $S_a$ is an injective $\Gamma$-module.

To connect the homological properties between $\operatorname{mod} \Lambda$ and $\operatorname{mod} \Gamma$, we employ the restriction and extension functors.

\begin{definition}
	\begin{enumerate}[label=\text{\rm(\arabic*)}]
		\item The \emph{extension functor} $\mathcal{E} \colon \operatorname{mod} \Lambda \to \operatorname{mod} \Gamma$ is defined by $\mathcal{E}(X) = (X, 0, 0)$.
		\item The \emph{restriction functor} $\mathcal{R} \colon \operatorname{mod} \Gamma \to \operatorname{mod} \Lambda$ is defined by $\mathcal{R}(X, V, \phi) = X$.
	\end{enumerate}
\end{definition}

\begin{lemma}\label{lem:adjunction}
	Let $\Gamma = \Lambda[M]$ be the one-point extension of $\Lambda$ by $M \in \operatorname{mod} \Lambda$. The extension functor $\mathcal{E}$ and the restriction functor $\mathcal{R}$ form an adjoint pair $(\mathcal{E}, \mathcal{R})$.
\end{lemma}

\begin{proof}
	Note that
	\[
	\operatorname{Hom}_\Gamma(\mathcal{E}X, Y) \cong \operatorname{Hom}_\Lambda(X, Y_0) \cong \operatorname{Hom}_\Lambda(X, \mathcal{R}Y),
	\]
	where $Y=(Y_0, V_0, \phi_0)$,	we obtain the desired result.
\end{proof}

The following Lemma  is well known.

\begin{lemma}\label{lem:restriction_basics}
	Let $\Gamma = \Lambda[M]$ be the one-point extension of $\Lambda$ by $M \in \operatorname{mod} \Lambda$. Let $\mathcal{R}$ and $\mathcal{E}$ be the restriction and extension functors, respectively. The following hold:
	\begin{enumerate}[label=\text{\rm(\arabic*)}]
		\item For any $X \in \operatorname{mod} \Lambda$, $\mathcal{R}(\mathcal{E}X) = X$.
		\item $\mathcal{R}S_a = 0$.
		\item  $\mathcal{R}P_a \cong M$.
	\end{enumerate}
\end{lemma}

\begin{lemma}\label{lem:ext_isomorphisms}
	Let $X, Y \in \operatorname{mod} \Lambda$. The following isomorphisms hold for all $i \ge 0$:
	\begin{enumerate}[label=\text{\rm(\arabic*)}]
		\item $\operatorname{Ext}^i_\Gamma(\mathcal{E}X, \mathcal{E}Y) \cong \operatorname{Ext}^i_\Lambda(X, Y)$.
		\item $\operatorname{Ext}^i_\Gamma(\mathcal{E}X, P_a) \cong \operatorname{Ext}^i_\Lambda(X, M)$.
		\item $\operatorname{Ext}^i_\Gamma(P_a, \mathcal{E}X) = 0$ for all $i \ge 1$.
	\end{enumerate}
\end{lemma}

\begin{proof}
	Recall that $(\mathcal{E}, \mathcal{R})$ is an adjoint pair of exact functors. Since $\mathcal{E}$ is exact, its right adjoint $\mathcal{R}$ preserves injective modules. Specifically, if $I^\bullet$ is an injective resolution of a $\Gamma$-module $Z$, then $\mathcal{R}I^\bullet$ is an injective resolution of $\mathcal{R}Z$ in $\operatorname{mod} \Lambda$.
	
	The adjunction $\operatorname{Hom}_\Gamma(\mathcal{E}X, -) \cong \operatorname{Hom}_\Lambda(X, \mathcal{R}(-))$ thus induces a natural isomorphism on the level of derived functors:
	\begin{equation} \label{eq1}
		\operatorname{Ext}^i_\Gamma(\mathcal{E}X, Z) \cong \operatorname{Ext}^i_\Lambda(X, \mathcal{R}Z), \quad \forall i \ge 0.
	\end{equation}
	
	(1)Take $Z = \mathcal{E}Y$. The isomorphism (\ref{eq1}) yields $\operatorname{Ext}^i_\Gamma(\mathcal{E}X, \mathcal{E}Y) \cong\operatorname{Ext}^i_\Lambda(X, \mathcal{RE}Y) \cong \operatorname{Ext}^i_\Lambda(X, Y)$.
	
	(2) Take $Z = P_a$.  We have  $\operatorname{Ext}^i_\Gamma(\mathcal{E}X, P_a) \cong \operatorname{Ext}^i_\Gamma(X,\mathcal{R} P_a) \cong\operatorname{Ext}^i_\Lambda(X, M)$ by Lemma \ref{lem:restriction_basics} (3) and the isomorphism (\ref{eq1}).
	
	Finally, (3) follows directly from the fact that $P_a$ is a projective object in $\operatorname{mod} \Gamma$.
\end{proof}

\section{Main results}\label{sec3}

In this section, we establish two methods for constructing projectively Wakamatsu tilting (PWT) modules over the one-point extension algebra $\Gamma = \Lambda[M]$ from a given PWT module over $\Lambda$.

	\begin{lemma}\label{lem:perp_category_proj}
	Let $U \in \operatorname{mod} \Lambda$ and $\widetilde{U} = \mathcal{E}U \oplus P_a$ in $\operatorname{mod} \Gamma$. For any module $Y \in \operatorname{mod} \Gamma$,
	\[
	Y \in \widetilde{U}^{\perp_\Gamma} \quad \text{if and only if} \quad \mathcal{R}Y \in U^{\perp_\Lambda}.
	\]
\end{lemma}

\begin{proof}
	Note that $P_a$ is a projective module in $\operatorname{mod} \Gamma$,
	we obtain
	\[
	\operatorname{Ext}^i_\Gamma(\widetilde{U}, Y) \cong \operatorname{Ext}^i_\Gamma(\mathcal{E}U, Y) \oplus \operatorname{Ext}^i_\Gamma(P_a, Y)\cong\operatorname{Ext}^i_\Gamma(\mathcal{E}U, Y) \cong \operatorname{Ext}^i_\Lambda(U, \mathcal{R}Y), \quad \text{for all } i > 0.
	\]
	The third isomorphism follows from Lemma \ref{lem:ext_isomorphisms}. Hence
	$\operatorname{Ext}^i_\Gamma(\widetilde{U}, Y) = 0$ for all $i > 0$ if and only if $\operatorname{Ext}^i_\Lambda(U, \mathcal{R}Y) = 0$ for all $i > 0$. This establishes the required equivalence.
\end{proof}

	Let $\mathcal{C}$ be a subcategory of $\operatorname{mod}\Lambda$. A morphism $f \colon X \to E$ with $E \in \mathcal{C}$ is called a \textit{left $\mathcal{C}$-approximation} of $X$ if any morphism from $X$ to an object in $\mathcal{C}$ factors through $f$. $f$ is called minimal if every $h\in\End E$ such that $hf=f$ is an automorphism.  Dually, we have the concept of right $\mathcal{C}$-approximation.
	
We first show that a PWT module over $\Lambda$ can be lifted to $\Gamma$ by adding the new projective module $P_a$.

\begin{theorem}\label{thm:main_construction_proj}
	Let $U$ be a PWT module in $\operatorname{mod} \Lambda$. If $M\in U^{\perp_\Lambda}$, then
	\[
	\widetilde{U} = \mathcal{E}U \oplus P_a
	\]
	is PWT module in $\operatorname{mod} \Gamma$.
\end{theorem}

\begin{proof}
	Using the additivity of the $\operatorname{Ext}$ functor, we have
	\[
	\operatorname{Ext}^i_\Gamma(\widetilde{U}, \widetilde{U}) \cong \operatorname{Ext}^i_\Gamma(\mathcal{E}U, \mathcal{E}U) \oplus \operatorname{Ext}^i_\Gamma(\mathcal{E}U, P_a) \oplus \operatorname{Ext}^i_\Gamma(P_a, \widetilde{U}).
	\]
	Since $P_a$ is a projective $\Gamma$-module, the $\operatorname{Ext}^i_\Gamma(P_a, \widetilde{U})$ vanishes for all $i > 0$.  $\operatorname{Ext}^i_\Gamma(\mathcal{E}U, \mathcal{E}U) \cong \operatorname{Ext}^i_\Lambda(U, U) = 0$ follows from Lemma \ref{lem:ext_isomorphisms}(1) and  $U$ is self-orthogonal in $\operatorname{mod} \Lambda$.
	By Lemma \ref{lem:ext_isomorphisms}(2), $\operatorname{Ext}^i_\Gamma(\mathcal{E}U, P_a) \cong \operatorname{Ext}^i_\Lambda(U, M)$ vanishes for all $i > 0$ by the hypothesis that $M \in U^{\perp_\Lambda}$.
	Consequently, $\operatorname{Ext}^{>0}_\Gamma(\widetilde{U}, \widetilde{U}) = 0$.
	
	Now, let $Y \in \widetilde{U}^{\perp_\Gamma}$.  Let $V_a$ be the vector space of $Y$ at the extension vertex $a$, and let $d = \dim_k V_a$. There is a canonical map $g \colon P_a^{\oplus d} \to Y$ covering the top of $Y$ at $a$.
	By Lemma \ref{lem:perp_category_proj}, the restriction $\mathcal{R}Y$ belongs to $U^{\perp_\Lambda}$. Since $U$ is an Ext-progenerator of $U^{\perp_\Lambda}$, there exists a surjection $f \colon U_0 \to \mathcal{R}Y$ with $U_0 \in \operatorname{add} U$.
	Applying the exact extension functor $\mathcal{E}$, we obtain a morphism $\mathcal{E}f \colon \mathcal{E}U_0 \to \mathcal{E}\mathcal{R}Y$. Define $h = (\mathcal{E}f, g) \colon \mathcal{E}U_0 \oplus P_a^{\oplus d} \to Y$, we have $h$ is surjective. Let $K = \operatorname{Ker} h$. This yields a short exact sequence in $\operatorname{mod} \Gamma$:
	\[
	0 \longrightarrow K \longrightarrow \widetilde{U}_0 \xrightarrow{h} Y \longrightarrow 0,
	\]
	where $\widetilde{U}_0 = \mathcal{E}U_0 \oplus P_a^{\oplus d} \in \operatorname{add} \widetilde{U}$. It remains to show $K \in \widetilde{U}^{\perp_\Gamma}$.
	Applying the exact restriction functor $\mathcal{R}$ and recalling that $\mathcal{R}P_a \cong M$, we obtain:
	\[
	0 \longrightarrow \mathcal{R}K \longrightarrow U_0 \oplus M^{\oplus d} \longrightarrow \mathcal{R}Y \longrightarrow 0.
	\]
 Applying the functor $\operatorname{Hom}_{\Lambda}(U, -)$ to this sequence yields the long exact sequence:
	\[
	\cdots \rightarrow \operatorname{Hom}_{\Lambda}(U, U_{0} \oplus M^{\oplus d}) \xrightarrow{(\mathcal{R}h)_{*}} \operatorname{Hom}_{\Lambda}(U, \mathcal{R}Y) \rightarrow \operatorname{Ext}_{\Lambda}^{1}(U, \mathcal{R}K) \rightarrow \operatorname{Ext}_{\Lambda}^{1}(U, U_{0} \oplus M^{\oplus d}) \rightarrow \cdots
	\]
	Notice that the restriction of $\mathcal{R}h$ to the $U_{0}$ summand is exactly $f$. Because $f$ is a right $\operatorname{add} U$-approximation, the induced map $f_{*}$ is surjective, which guarantees the surjectivity of $(\mathcal{R}h)_{*}$. Moreover, since both $U_{0}$ and $M$ belong to $U^{\perp_{\Lambda}}$, the middle term $U_{0} \oplus M^{\oplus d}$ is in $U^{\perp_{\Lambda}}$, meaning $\operatorname{Ext}_{\Lambda}^{1}(U, U_{0} \oplus M^{\oplus d}) = 0$. The exactness of the sequence thus forces $\operatorname{Ext}_{\Lambda}^{1}(U, \mathcal{R}K) = 0$.	For $i > 0$, dimension shifting along the exact sequence gives $\operatorname{Ext}_{\Lambda}^{i+1}(U, \mathcal{R}K) \cong \operatorname{Ext}_{\Lambda}^{i}(U, \mathcal{R}Y) = 0$, since $\mathcal{R}Y \in U^{\perp_{\Lambda}}$. Therefore, $\operatorname{Ext}_{\Lambda}^{>0}(U, \mathcal{R}K) = 0$, establishing that $\mathcal{R}K \in U^{\perp_{\Lambda}}$.
	By Lemma \ref{lem:perp_category_proj} again, this implies $K \in 	\widetilde{U}^{\perp_{\Gamma}}$. This completes the proof.
\end{proof}

\begin{definition} \label{def:mutation}
	Let $U=\overline{U}\oplus X$ be a basic PWT module in $\mathrm{mod}\,\Lambda$, where $X$ is an indecomposable direct summand. A basic PWT module $U^{\prime}$ is called a left mutation of $U$ with respect to $X$ if $U^{\prime}=\overline{U}\oplus Y$ for some indecomposable module $Y\not\cong X$, and there exists a short exact sequence in $\mathrm{mod}\,\Lambda$:
	$$0\longrightarrow X\xrightarrow{f} E\longrightarrow Y\longrightarrow 0$$
	where $f$ is a minimal left $\mathrm{add}\,\overline{U}$-approximation of $X$.
\end{definition}

Next, we will prove that this lifting process in Theorem \ref{thm:main_construction_proj} perfectly preserves  mutation relations under certain homological conditions. 

\begin{lemma}\label{lem:approx_lifting}	Let $\Gamma = \Lambda[M]$ be the one-point extension of $\Lambda$ by $M \in \operatorname{mod} \Lambda$ and  $U \in \mathrm{mod}\,\Lambda$. Consider a short exact sequence in $\mathrm{mod}\,\Lambda$:
	$$ 0 \longrightarrow X \xrightarrow{f} E \longrightarrow Y \longrightarrow 0 $$
	where $f$ is a left $\mathrm{add}\,U$-approximation of $X$. Then the induced morphism $\mathcal{E}f \colon \mathcal{E}X \rightarrow \mathcal{E}E$ is a left $\mathrm{add}(\mathcal{E}U \oplus P_{a})$-approximation of $\mathcal{E}X$ in $\mathrm{mod}\,\Gamma$ if and only if the induced map $\mathrm{Hom}_{\Lambda}(E,M) \rightarrow \mathrm{Hom}_{\Lambda}(X,M)$ is surjective. In particular, this holds if $\mathrm{Ext}_{\Lambda}^{1}(Y,M) = 0$. Moreover, $f$ is minimal if and only if $\mathcal{E}f$ is also.
\end{lemma}

\begin{proof}Take  $\widetilde{\mathcal{C}} = \mathrm{add}(\mathcal{E}U \oplus P_{a})$,
	$\mathcal{E}f$ is a left $\widetilde{\mathcal{C}}$-approximation if and only if any morphism $h \colon \mathcal{E}X \rightarrow \widetilde{C}$ with $\widetilde{C} \in \widetilde{\mathcal{C}}$ factors through $\mathcal{E}f$. Any object in $\widetilde{\mathcal{C}}$ is of the form $\mathcal{E}U_{0} \oplus P_{a}^{\oplus k}$ for some $U_{0} \in \mathcal{C}$ and $k \ge 0$.
	Since $f$ is a left $\mathcal{C}$-approximation in $\mathrm{mod}\,\Lambda$ and the exact extension functor $\mathcal{E}$ is fully faithful, any morphism $\mathcal{E}X \rightarrow \mathcal{E}U_{0}$ factors through $\mathcal{E}f$. Therefore, we only need to consider factorizations of morphisms targeting the projective summand $P_{a}$.
	By the standard adjunction isomorphism, $\mathrm{Hom}_{\Gamma}(\mathcal{E}X, P_{a}) \cong \mathrm{Hom}_{\Lambda}(X, M)$ and $\mathrm{Hom}_{\Gamma}(\mathcal{E}E, P_{a}) \cong \mathrm{Hom}_{\Lambda}(E, M)$. Thus, the factorization property in $\mathrm{mod}\,\Gamma$ is precisely equivalent to the surjectivity of the naturally induced map $\mathrm{Hom}_{\Lambda}(E,M) \rightarrow \mathrm{Hom}_{\Lambda}(X,M)$.
	Applying the covariant functor $\mathrm{Hom}_{\Lambda}(-, M)$ to the short exact sequence in $\mathrm{mod}\,\Lambda$, we obtain the long exact sequence:
	$$ \cdots \longrightarrow \mathrm{Hom}_{\Lambda}(E, M) \longrightarrow \mathrm{Hom}_{\Lambda}(X, M) \longrightarrow \mathrm{Ext}_{\Lambda}^{1}(Y, M) \longrightarrow \mathrm{Ext}_{\Lambda}^{1}(E, M) \longrightarrow \cdots $$
	The exactness of the sequence forces the map $\mathrm{Hom}_{\Lambda}(E, M) \rightarrow \mathrm{Hom}_{\Lambda}(X, M)$ to be surjective if $\mathrm{Ext}_{\Lambda}^{1}(Y, M) = 0$. 
The preservation of the minimal left approximation strictly follows from the fact that the exact extension functor $\mathcal{E}$ is fully faithful, which guarantees that any endomorphism $g \in \operatorname{End}_\Gamma(\mathcal{E}E)$ satisfying $g \circ \mathcal{E}f = \mathcal{E}f$ is of the form $\mathcal{E}h$ for some automorphism $h \in \operatorname{End}_\Lambda(E)$, making $g$ an automorphism as well, thus completing the proof.
\end{proof}

\begin{theorem} \label{thm:mutation_preservation}
	Let $\Gamma = \Lambda[M]$ be the one-point extension of $\Lambda$ by $M \in \operatorname{mod} \Lambda$. 
	Suppose $U$  is a PWT module over $\Lambda$ and $U^{\prime}$ is a left mutation of $U$. If $M\in U^{\perp_{\Lambda}}\cap (U^{\prime})^{\perp_{\Lambda}}$, then
	$\widetilde{U}^{\prime}=\mathcal{E}U^{\prime}\oplus P_{a}$ is a left mutation of $\widetilde{U}=\mathcal{E}U\oplus P_{a}$ in $\mathrm{mod}\,\Gamma$.
\end{theorem}

\begin{proof}
	Let $U=\overline{U}\oplus X$ be a PWT module over $\Lambda$ and  $U^{\prime}=\overline{U}\oplus Y$  a left mutation of $U$. Then
	$$\widetilde{U}=\mathcal{E}\overline{U}\oplus \mathcal{E}X\oplus P_{a} \quad \text{and} \quad \widetilde{U}^{\prime}=\mathcal{E}\overline{U}\oplus \mathcal{E}Y\oplus P_{a}.$$
Theorem \ref{thm:main_construction_proj} follows $\widetilde{U}$ and $\widetilde{U}^{\prime}$ are PWT module in $\mathrm{mod}\,\Gamma$.
	Consequently, the common direct summand of them is precisely $\widetilde{U}_{\mathrm{com}}=\mathcal{E}\overline{U}\oplus P_{a}$. 
	Next, We will show that $\widetilde{U}^{\prime}$ is obtained from $\widetilde{U}$ by mutating the indecomposable summand $\mathcal{E}X$. 
	Applying the exact functor $\mathcal{E}$ to the mutation sequence in $\mathrm{mod}\,\Lambda$, we obtain the following exact sequence in $\mathrm{mod}\,\Gamma$:
	$$0\longrightarrow \mathcal{E}X\xrightarrow{\mathcal{E}f} \mathcal{E}E\longrightarrow \mathcal{E}Y\longrightarrow 0.$$
	Since $E\in \mathrm{add}\,\overline{U}$, it follows that $\mathcal{E}E\in \mathrm{add}(\mathcal{E}\overline{U})\subseteq \mathrm{add}\,\widetilde{U}_{\mathrm{com}}$.
	By our hypothesis, $M\in (U^{\prime})^{\perp_{\Lambda}}$ which implies $\mathrm{Ext}_{\Lambda}^{1}(Y,M)=0$.
	By Lemma \ref{lem:approx_lifting}, the morphism $\mathcal{E}f$ is a minimal left $\mathrm{add}\,\widetilde{U}_{\mathrm{com}}$-approximation of $\mathcal{E}X$.  We conclude that $\widetilde{U}^{\prime}$ is precisely the left mutation of $\widetilde{U}$ with respect to $\mathcal{E}X$.
\end{proof}

Inspired by \cite[Theorem 2.1]{Gao22}, we  focus on a specific class of one-point extensions, establishing a bijection for PWT modules between $\mathrm{mod}\,\Lambda$ and $\mathrm{mod}\,\Gamma$.

\begin{definition} \label{def:source_point}
	Let $\Lambda = kQ/I$ be a finite-dimensional basic algebra. A vertex $i \in Q_0$ is called a source point if there are no arrows in $Q_1$ ending at $i$. It is a classical fact that $i$ is a source point if and only if the corresponding simple module $S_i$ is an injective $\Lambda$-module.
	A \emph{source point extension} of $\Lambda$ is defined as the one-point extension algebra $\Gamma = \Lambda[S_i]$.
\end{definition}

In a source point extension $\Gamma = \Lambda[S_i]$, the new indecomposable projective $\Gamma$-module $P_a$ is explicitly a projective-injective module. Moreover, all indecomposable 
$\Gamma$-modules exactly consist of  $P_a$, $S_a$ and the  image of all indecomposable 
$\Lambda$-modules under  the extension functor  $\mathcal{E}$,.

\begin{lemma}\label{lem:self-ortho}
Let $\Gamma=\Lambda[S_i]$ be a source point extension and $U' \in \operatorname{mod} \Lambda$  a self-orthogonal module. If $\Hom_\Lambda(S_i, U') = 0$ and  $\operatorname{Ext}_\Lambda^{>0}(S_i, U') = 0$,
then the $\Gamma$-module $\widetilde{U} = \mathcal{E}U' \oplus P_a \oplus S_a$ is  self-orthogonal in $\operatorname{mod} \Gamma$.
\end{lemma}

\begin{proof}  Since $P_a$ is a projective $\Gamma$-module and $S_a$ is an injective $\Gamma$-module, we have $\operatorname{Ext}_\Gamma^j(P_a, -) = 0$  and $\operatorname{Ext}_\Lambda^j(U', S_i) = 0$ for all $j > 0$.	Thus, 
\[
\operatorname{Ext}^i_\Gamma(\widetilde{U}, \widetilde{U}) \cong \operatorname{Ext}^i_\Gamma(\mathcal{E}U', \mathcal{E}U') \oplus \operatorname{Ext}^i_\Gamma(\mathcal{E}U', P_a) \oplus \operatorname{Ext}^i_\Gamma(S_a,\mathcal{E}U')\oplus \operatorname{Ext}^i_\Gamma(S_a,P_a).
\]	
By the self-orthogonality of $U'$, we obtain $\operatorname{Ext}_\Gamma^j(\mathcal{E}U', \mathcal{E}U') \cong \operatorname{Ext}_\Lambda^j(U', U') = 0$. Lemma \ref{lem:ext_isomorphisms} implies	$\operatorname{Ext}_\Gamma^j(\mathcal{E}U', P_a) \cong \operatorname{Ext}_\Lambda^j(U', S_i) = 0$. Applying the functor $\operatorname{Hom}_\Gamma(-, \mathcal{E}U')$ to the canonical short exact sequence $0 \to \mathcal{E}S_i \to P_a \to S_a \to 0$ yields:
$$ \operatorname{Hom}_\Gamma(P_a, \mathcal{E}U') \to \operatorname{Hom}_\Gamma(\mathcal{E}S_i, \mathcal{E}U') \to \operatorname{Ext}_\Gamma^1(S_a, \mathcal{E}U') \to\operatorname{Ext}_\Gamma^1(P_a, \mathcal{E}U')=0 . $$
Since $P_a$ is the projective cover of $S_a$ at vertex $a$, and the module $\mathcal{E}U'$ is supported entirely on the subalgebra $\Lambda$ (meaning its vector space at vertex $a$ is zero), any morphism $P_a \to \mathcal{E}U'$ must be zero. Thus, $\operatorname{Hom}_\Gamma(P_a, \mathcal{E}U') = 0$.
This immediately yields an isomorphism $\operatorname{Ext}_\Gamma^1(S_a, \mathcal{E}U') \cong \operatorname{Hom}_\Gamma(\mathcal{E}S_i, \mathcal{E}U') \cong \operatorname{Hom}_\Lambda(S_i, U')$, which vanishes by our assumption.
For $j >0 $, dimension shifting gives $\operatorname{Ext}_\Gamma^{j+1}(S_a, \mathcal{E}U') \cong \operatorname{Ext}_\Gamma^j(\mathcal{E}S_i, \mathcal{E}U') \cong \operatorname{Ext}_\Lambda^j(S_i, U')$, which explicitly vanishes by our assumption again. This implies $\operatorname{Ext}^i_\Gamma(S_a,\mathcal{E}U')=0$ for all  $i >0 $. Similarly, applying the functor $\operatorname{Hom}_\Gamma(-, P_a)$ to the same sequence yields:
$$ 0 =\operatorname{Hom}_\Gamma(S_a, P_a) \to \operatorname{Hom}_\Gamma(P_a, P_a) \xrightarrow{\alpha} \operatorname{Hom}_\Gamma(\mathcal{E}S_i, P_a) \to \operatorname{Ext}_\Gamma^1(S_a, P_a) \to 0. $$
Since $P_a$ is indecomposable, $\operatorname{End}_\Gamma(P_a) \cong k$. By adjunction, $\operatorname{Hom}_\Gamma(\mathcal{E}S_i, P_a) \cong \operatorname{Hom}_\Lambda(S_i, \mathcal{R}P_a) \cong \operatorname{Hom}_\Lambda(S_i, S_i) \cong k$. The   map $\alpha$ must necessarily be an isomorphism. Hence, its cokernel $\operatorname{Ext}_\Gamma^1(S_a, P_a) = 0$.
For $j >0 $, dimension shifting gives $\operatorname{Ext}_\Gamma^{j+1}(S_a, P_a)  \cong \operatorname{Ext}_\Gamma^j(\mathcal{E}S_i, P_a)\cong \operatorname{Ext}_\Lambda^j(S_i, S_i)=0$ since $S_i$ is an injective $\Lambda$-module. Finally, all component extensions vanish, proving that $\operatorname{Ext}_\Gamma^{>0}(\widetilde{U}, \widetilde{U}) = 0$.
\end{proof}

\begin{proposition}\label{3.5}
Let $\Gamma=\Lambda[S_i]$ be a source point extension of a representation-finite algebra $\Lambda$	
and $U$   a  PWT module in $\operatorname{mod} \Lambda$. If the simple injective module $S_i$ is a direct summand of $U$, we decompose $U = U' \oplus S_i$. Then the $\Gamma$-module
$$ \widetilde{U} = \mathcal{E}U' \oplus P_a \oplus S_a $$
is a PWT module in $\operatorname{mod} \Gamma$.
\end{proposition}

\begin{proof}
Since $U = U' \oplus S_i$ is a PWT module over $\Lambda$. This directly implies that $\operatorname{Ext}_\Lambda^{>0}(U', U') = 0$ and $\operatorname{Ext}_\Lambda^{>0}(S_i, U') = 0$.
Furthermore, since $U'$ does not contain $S_i$ as a direct summand and $S_i$ is a simple module, any non-zero morphism $S_i \to U'$ must be a monomorphism. Since $S_i$ is also an injective $\Lambda$-module, this monomorphism must split, which implies $S_i$ is  a direct summand of $U'$. This contradiction ensures that $\operatorname{Hom}_\Lambda(S_i, U') = 0$. By Lemma \ref{lem:self-ortho}, the $\Gamma$-module $\widetilde{U} = \mathcal{E}U' \oplus P_a \oplus S_a$ is  self-orthogonal in $\operatorname{mod} \Gamma$. According to Lemma \ref{lem:pwt_max_orthogonal}, for representation-finite algebras, the class of PWT modules coincides with maximal self-orthogonal modules. Therefore $|U|=|\Lambda|-1$.
Note that 
$$ |\widetilde{U}| = |\mathcal{E}U'| + |P_a| + |S_a|= |U'| + 2 =(|U| - 1) +2 = |\Lambda| + 1 = |\Gamma|. $$
We have $\widetilde{U}$ is a maximal self-orthogonal module.  $\Gamma$ is also representation-finite algebra, and thus $\widetilde{U}$  is a PWT module in $\operatorname{mod} \Gamma$.
\end{proof}

We now state the main bijection theorem. Let $\mathrm{PWT}(-)$ denote the set of isomorphism classes of basic PWT modules. We define $\mathrm{PWT}(\Lambda, S_i)$ as the subset of $\mathrm{PWT}(\Lambda)$ consisting of all basic PWT modules over $\Lambda$ that explicitly contain $S_i$ as a direct summand. Now, let $\mathrm{RPWT}(\Lambda, S_i)$ be a set consisting of all basic modules formed by removing the direct summand $S_i$ from modules in $\mathrm{PWT}(\Lambda, S_i)$. Clearly, if $\Lambda$ is a representation-finite  algebra, then
$$\mathrm{RPWT}(\Lambda, S_i)=\{U'\in\mod\Lambda||U'|=|\Lambda|-1,U'\oplus S_i\in\mathrm{PWT}(\Lambda, S_i)\}.$$

\begin{theorem} \label{thm:bijection}
Let $\Gamma=\Lambda[S_i]$ be a source point extension of a representation-finite  algebra $\Lambda$. There is a bijection:
$$ \mathrm{PWT}(\Gamma) \longleftrightarrow \mathrm{PWT}(\Lambda) \coprod \mathrm{RPWT}(\Lambda, S_i). $$
Consequently, the numbers of  PWT modules satisfy:
$$ |\mathrm{PWT}(\Gamma)| = |\mathrm{PWT}(\Lambda)| + |\mathrm{RPWT}(\Lambda, S_i)|= |\mathrm{PWT}(\Lambda)| + |\mathrm{PWT}(\Lambda, S_i)|. $$
\end{theorem}

\begin{proof}Firstly,
for any $U \in \mathrm{PWT}(\Lambda)$,  $\mathrm{Ext}_\Lambda^{>0}(U, S_i) = 0$ holds  since $S_i$ is an injective $\Lambda$-module. By Theorem \ref{thm:main_construction_proj}, we map $U \mapsto \mathcal{E}U \oplus P_a$, which is a PWT module in $\mathrm{mod}\,\Gamma$. For any $U' \in \mathrm{RPWT}(\Lambda, S_i)$, we have $U' \oplus S_i \in \mathrm{PWT}(\Lambda, S_i)$. By Theorem \ref{3.5}, we map $U' \mapsto \mathcal{E}U' \oplus P_a \oplus S_a$, which is also a PWT module in $\mathrm{mod}\,\Gamma$. Since the modules in the first class do not contain the simple injective $\Gamma$-module $S_a$ as their direct summand, while those in the second class explicitly do, the images of the two maps are strictly disjoint. Thus, the forward mapping is well-defined and injective.

Secondly, Let $\widetilde{T} \in \mathrm{PWT}(\Gamma)$. Since  $\Lambda$ is representation-finite, we have $\Gamma$ is also representation-finite and $\widetilde{T}$ is a maximal self-orthogonal module, meaning $|\widetilde{T}| = |\Gamma| = |\Lambda| + 1$.
 Note that $P_a$ a projective-injective $\Gamma$-module, Remark \ref{rem:proj_inj_summands} implies any PWT module must contain all projective-injective modules as direct summands. Hence, $P_a$ is a direct summand of $\widetilde{T}$. Any other summand of $\widetilde{T}$ is either $S_a$ or belongs to $\mathcal{E}(\mathrm{mod}\,\Lambda)$. 

\textbf{Case 1:} $S_a \notin \mathrm{add}\,\widetilde{T}$.
Then $\widetilde{T} = \mathcal{E}U \oplus P_a$ for some basic $\Lambda$-module $U$. Since $\widetilde{T}$ is self-orthogonal in $\mathrm{mod}\,\Gamma$, $\mathcal{E}U$ is self-orthogonal in $\mathrm{mod}\,\Gamma$ which implies  $U$ must be self-orthogonal in $\mathrm{mod}\,\Lambda$. Moreover, $|U| = |\mathcal{E}U| = |\widetilde{T}| - 1 = |\Lambda|$. This implies $U$ is a maximal self-orthogonal module, hence $U \in \mathrm{PWT}(\Lambda)$.

\textbf{Case 2:} $S_a \in \mathrm{add}\,\widetilde{T}$.
Then $\widetilde{T} = \mathcal{E}U' \oplus P_a \oplus S_a$ for some basic $\Lambda$-module $U'$. The self-orthogonality of $\widetilde{T}$ implies $\mathrm{Ext}_\Lambda^{>0}(U', U') = 0$.
Applying the functor $\mathrm{Hom}_\Gamma(-, \mathcal{E}U')$ to the canonical short exact sequence $0 \to \mathcal{E}S_i \to P_a \to S_a \to 0$ yields a long exact sequence. Since $P_a$ is projective, $\mathrm{Ext}_\Gamma^{>0}(P_a, \mathcal{E}U') = 0$. By dimension shifting, $\mathrm{Ext}_\Lambda^j(S_i, U') \cong \mathrm{Ext}_\Gamma^j(\mathcal{E}S_i, \mathcal{E}U') \cong \mathrm{Ext}_\Gamma^{j+1}(S_a, \mathcal{E}U')$. Because $\widetilde{T}$ is self-orthogonal, $\mathrm{Ext}_\Gamma^{j+1}(S_a, \mathcal{E}U') = 0$ for all $j >0$, which strictly forces $\mathrm{Ext}_\Lambda^{>0}(S_i, U') = 0$. Moreover, for $j=0$, the exact sequence yields:
$$ 0 \longrightarrow \mathrm{Hom}_\Gamma(S_a, \mathcal{E}U') \longrightarrow \mathrm{Hom}_\Gamma(P_a, \mathcal{E}U') \longrightarrow \mathrm{Hom}_\Gamma(\mathcal{E}S_i, \mathcal{E}U') \longrightarrow \mathrm{Ext}_\Gamma^1(S_a, \mathcal{E}U') \longrightarrow 0. $$
Note that  $\mathcal{E}U'$ is supported entirely on $\Lambda$,  $P_a$  has no non-zero morphisms to $\mathcal{E}U'$, meaning $\mathrm{Hom}_\Gamma(P_a, \mathcal{E}U') = 0$. Since $\widetilde{T}$ is self-orthogonal, $\mathrm{Ext}_\Gamma^1(S_a, \mathcal{E}U') = 0$. So, we have $\mathrm{Hom}_\Lambda(S_i, U') \cong \mathrm{Hom}_\Gamma(\mathcal{E}S_i, \mathcal{E}U') = 0$.
Now consider the $\Lambda$-module $T = U' \oplus S_i$. Because $\mathrm{Hom}_\Lambda(S_i, U') = 0$, $S_i$ is not a direct summand of $U'$, making $T$ a strictly basic module. It is easy to verify its self-orthogonality in $\mathrm{mod}\,\Lambda$. Moreover, $|T| = |U'| + 1 = (|\widetilde{T}| - 2) + 1 = |\Lambda|$. Thus, $T$ is a maximal self-orthogonal module in $\mathrm{mod}\,\Lambda$ that explicitly contains $S_i$. Therefore, $T \in \mathrm{PWT}(\Lambda, S_i)$ which shows $U' \in \mathrm{RPWT}(\Lambda, S_i)$. 

 It is immediately clear that the forward mapping is also injective

Finally, both $\mathrm{PWT}(\Gamma)$ and  $\mathrm{PWT}(\Lambda)$  are finite sets, establishing the bijection and confirming the combinatorial formula.\end{proof}


\begin{remark}\label{rem:comparison_ice}
In classical tilting theory, the number of tilting modules does not increase under source point extensions; specifically, $|\mathrm{tilt}\,\Gamma| = |\mathrm{tilt}\,\Lambda|$ \cite{AHT07, Gao22}. This is because, in the extension $\Gamma = \Lambda[S_i]$, the new simple module $S_a$ generally has projective dimension $\mathrm{pd}_\Gamma S_a \ge 2$ (unless $S_i$ is projective). Limiting the projection dimension to at most one in the bijection of Theorem \ref{thm:bijection}, we also obtain the bijection between $\mathrm{tilt}\,\Gamma = \mathrm{tilt}\,\Lambda$.
\end{remark}

\section{Examples}\label{sec4}

In this section, we provide examples to illustrate our results.

\begin{example} \label{ex:4.1}
Let $\Lambda$ be the Nakayama algebra given by the quiver $2 \xrightarrow{\alpha} 3 \xrightarrow{\beta} 4$ with the relation $\alpha\beta=0$. The indecomposable modules over $\Lambda$ are:
$$ P_2 = \begin{smallmatrix} 2 \\ 3 \end{smallmatrix}, \quad P_3 = \begin{smallmatrix} 3 \\ 4\end{smallmatrix}, \quad P_4 = \begin{smallmatrix} 4 \end{smallmatrix}, \quad S_2 = \begin{smallmatrix} 2  \end{smallmatrix}, \quad S_3 = \begin{smallmatrix} 3 \end{smallmatrix}. $$
 Since $\Lambda$ is representation-finite, PWT modules coincide with maximal self-orthogonal modules. By Remark \ref{rem:proj_inj_summands}, any PWT module over $\Lambda$ must contain all projective-injective modules, meaning $P_2$ and $P_3$ must appear as direct summands in every PWT module. Thus, any PWT module over $\Lambda$ is of the form $ \begin{smallmatrix} 2 \\ 3 \end{smallmatrix} \oplus \begin{smallmatrix} 3 \\ 4 \end{smallmatrix} \oplus X$ for some indecomposable module $X \in \{ \begin{smallmatrix} 4 \end{smallmatrix},  \begin{smallmatrix} 3 \end{smallmatrix},  \begin{smallmatrix} 2  \end{smallmatrix}\}$.
It is straightforward to check that all three choices yield strictly self-orthogonal modules, giving exactly 3 basic PWT modules over $\Lambda$:
$$ U_1 = \begin{smallmatrix} 2 \\ 3 \end{smallmatrix} \oplus \begin{smallmatrix} 3 \\ 4 \end{smallmatrix}  \oplus \begin{smallmatrix} 4  \end{smallmatrix},\quad U_2 = \begin{smallmatrix} 2 \\ 3 \end{smallmatrix} \oplus \begin{smallmatrix} 3 \\ 4 \end{smallmatrix}  \oplus \begin{smallmatrix} 3 \end{smallmatrix}, \quad U_3 = \begin{smallmatrix} 2 \\ 3 \end{smallmatrix} \oplus \begin{smallmatrix} 3 \\ 4 \end{smallmatrix}  \oplus \begin{smallmatrix} 2  \end{smallmatrix}. $$
Thus, $|\mathrm{PWT}(\Lambda)| = 3$ and  $\mathrm{PWT}(\Lambda, S_2) = \{U_3\}$, yielding $\mathrm{RPWT}(\Lambda, S_2) =\{\begin{smallmatrix} 2 \\ 3 \end{smallmatrix} \oplus \begin{smallmatrix} 3 \\ 4 \end{smallmatrix}\}$.

Now, let $\Gamma = \Lambda[S_2]$ be the source point extension of $\Lambda$ by the simple injective module $M = S_2 = \begin{smallmatrix} 2  \end{smallmatrix}$. Then  $\Gamma$ is the algebra given by the quiver $1 \xrightarrow{\gamma} 2 \xrightarrow{\alpha} 3 \xrightarrow{\beta} 4$ with relations $\gamma\alpha=0$ and $\alpha\beta=0$.
The new indecomposable projective $\Gamma$-module at vertex $1$ is $P_1 = \begin{smallmatrix} 1 \\ 2 \end{smallmatrix}$, and the new simple injective module is $S_1 = \begin{smallmatrix} 1 \end{smallmatrix}$. Theorem \ref{thm:bijection} implies
$$ |\mathrm{PWT}(\Gamma)| = |\mathrm{PWT}(\Lambda)| + |\mathrm{RPWT}(\Lambda, S_2)| = 3 + 1 = 4. $$
These 4 PWT $\Gamma$-modules is as follows:
$$ \widetilde{U}_1 = \mathcal{E}U_1 \oplus P_1 = \begin{smallmatrix} 2 \\ 3 \end{smallmatrix} \oplus \begin{smallmatrix} 3 \\ 4 \end{smallmatrix} \oplus\begin{smallmatrix}4 \end{smallmatrix} \oplus \begin{smallmatrix} 1 \\ 2 \end{smallmatrix}, $$
$$\widetilde{U}_2 = \mathcal{E}U_2 \oplus P_1 =\begin{smallmatrix} 2 \\ 3 \end{smallmatrix} \oplus \begin{smallmatrix} 3 \\ 4 \end{smallmatrix} \oplus \begin{smallmatrix}  3 \end{smallmatrix} \oplus \begin{smallmatrix} 1 \\ 2 \end{smallmatrix}, $$
  $$  \widetilde{U}_3 = \mathcal{E}U_3 \oplus P_1 = \begin{smallmatrix} 2 \\ 3 \end{smallmatrix} \oplus \begin{smallmatrix} 3 \\ 4 \end{smallmatrix}\oplus \begin{smallmatrix} 2 \end{smallmatrix} \oplus \begin{smallmatrix} 1 \\ 2 \end{smallmatrix},$$
  $$ \widetilde{U}_4 =\begin{smallmatrix} 2 \\ 3 \end{smallmatrix} \oplus \begin{smallmatrix} 3 \\ 4 \end{smallmatrix} \oplus P_1\oplus S_1=\begin{smallmatrix} 2 \\ 3 \end{smallmatrix} \oplus \begin{smallmatrix} 3 \\ 4 \end{smallmatrix} \oplus \begin{smallmatrix} 1 \\ 2 \end{smallmatrix} \oplus \begin{smallmatrix} 1 \end{smallmatrix}. $$

\end{example}

\begin{example} \label{ex:4.2}
We show that the mutation preservation under the same source point extension $\Gamma = \Lambda[S_2]$ as in Example \ref{ex:4.1}.
Consider the  PWT module over $\Lambda$:
$$ U_1 = \begin{smallmatrix} 2 \\ 3 \end{smallmatrix} \oplus \begin{smallmatrix} 3 \\ 4 \end{smallmatrix}\oplus \begin{smallmatrix} 4\end{smallmatrix}, $$
and mutate it with respect to the indecomposable summand $\begin{smallmatrix} 4\end{smallmatrix}$.
The minimal left $\mathrm{add}(\begin{smallmatrix} 2 \\ 3 \end{smallmatrix} \oplus \begin{smallmatrix} 3 \\ 4 \end{smallmatrix})$-approximation of $\begin{smallmatrix} 4\end{smallmatrix}$ in $\mathrm{mod}\,\Lambda$ is exactly the  monomorphism $\begin{smallmatrix} 4\end{smallmatrix} \hookrightarrow \begin{smallmatrix} 3\\4\end{smallmatrix}$. The short exact sequence is:
$$ 0 \longrightarrow \begin{smallmatrix} 4\end{smallmatrix} \longrightarrow \begin{smallmatrix} 3\\4\end{smallmatrix} \longrightarrow \begin{smallmatrix} 3\end{smallmatrix} \longrightarrow 0. $$
Replacing the summand $\begin{smallmatrix} 4\end{smallmatrix}$ with the cokernel $\begin{smallmatrix} 3\end{smallmatrix}$, we obtain the left mutated module:
$$ U_2 = \begin{smallmatrix} 2\\3\end{smallmatrix} \oplus \begin{smallmatrix} 3\\4\end{smallmatrix} \oplus \begin{smallmatrix} 3\end{smallmatrix}. $$
Since $M = S_2 = 2$ is an injective $\Lambda$-module, the condition $\mathrm{Ext}_\Lambda^{>0}(-, S_2) = 0$ holds  for all modules. Consequently, the homological condition $M \in U_1^{\perp_\Lambda} \cap (U_2)^{\perp_\Lambda}$ required by Theorem \ref{thm:mutation_preservation} is satisfied. According to Theorem \ref{thm:mutation_preservation}, the  $\widetilde{U}_2$ must be a left mutation of $\widetilde{U}_1$ in $\mathrm{mod}\,\Gamma$. Their common direct summand is $\widetilde{U}_{\mathrm{com}} = \begin{smallmatrix} 2\\3\end{smallmatrix} \oplus \begin{smallmatrix} 3\\4\end{smallmatrix} \oplus \begin{smallmatrix} 1\\2\end{smallmatrix}.$
 The minimal left $\mathrm{add}(\widetilde{U}_{\mathrm{com}})$-approximation of
$\begin{smallmatrix} 4\end{smallmatrix} $ 
in $\mathrm{mod}\,\Gamma$  remains the  same monomorphism $\begin{smallmatrix} 4\end{smallmatrix} \hookrightarrow \begin{smallmatrix} 3\\4\end{smallmatrix}$, because there are no non-zero morphisms from the module $\begin{smallmatrix} 4\end{smallmatrix} $  to the new projective summand $\begin{smallmatrix} 1\\2\end{smallmatrix} $. Consequently, the cokernel remains $\begin{smallmatrix} 3\end{smallmatrix} $.
\end{example}


\begin{thebibliography}{99}
\bibitem[AIR14]{AIR14}
T. Adachi, O. Iyama, and I. Reiten,
\textit{$\tau$-tilting theory},
Compos. Math. \textbf{150} (2014), no. 3, 415--452.

\bibitem[AR91]{AR91}
M. Auslander and I. Reiten,
\textit{Applications of contravariantly finite subcategories},
Adv. Math. \textbf{86} (1991), no. 1, 111--152.


\bibitem[AHT07]{AHT07}
I. Assem, D. Happel, and S. Trepode,
\textit{Extending tilting modules to one-point extensions by projectives},
Comm. Algebra \textbf{35} (2007), no. 10, 2983--3006.



\bibitem[BS98]{BS98}
A. B. Buan, \O. Solberg,
\textit{Relative cotilting theory and almost complete cotilting modules},
Algebras and modules, II (Geiranger, 1996), 77--92, CMS Conf. Proc., \textbf{24}, Amer. Math. Soc., Providence, RI, 1998.

\bibitem[Eno23]{Eno23}
H. Enomoto,
\textit{Maximal self-orthogonal modules and a new generalization of tilting modules},
arXiv:2301.13498v2 [math.RT], 2023.

\bibitem[Gao22]{Gao22}
H. Gao,
\textit{$\tau$-tilting modules over one-point extensions by a simple module at a source point},
J. Algebra Appl. \textbf{21} (2022), no. 6,  Paper No. 2250122, 8 pp.


\bibitem[Sua18]{Sua18}
P. Suarez,
\textit{$\tau$-tilting modules over one-point extensions by a projective module},
Algebr. Represent. Theory  \textbf{21} (2018), no. 4, 769--786.

\bibitem[Wa90]{Wa90}
T. Wakamatsu,
\textit{Stable equivalence for self-injective algebras and a generalization of tilting modules},
J. Algebra \textbf{134} (1990), no. 2, 298--325.
\end{thebibliography}
\end{document}